\documentclass[a4paper]{amsart}

\usepackage{amsmath, amssymb, latexsym, amsfonts, pigpen, enumitem}
\usepackage[matrix,arrow,curve]{xy}

\usepackage[top=0.8in, left=0.95in, right=0.95in, bottom=1.5in]{geometry}

\numberwithin{equation}{section}
\theoremstyle{plain}
\newtheorem{theorem}[equation]{Theorem}
\newtheorem{corollary}[equation]{Corollary}
\newtheorem{proposition}[equation]{Proposition}
\newtheorem{lemma}[equation]{Lemma}

\theoremstyle{definition}
\newtheorem{definition}[equation]{Definition}
\newtheorem{notation}[equation]{Notation}
\theoremstyle{remark}
\newtheorem{remark}[equation]{Remark}
\newtheorem{example}[equation]{Example}

\newcommand{\A}{\mathbb{A}}
\renewcommand{\P}{\mathbb{P}}

\newcommand{\PP}{\mathbb{P}}

\newcommand{\Cone}{\operatorname{Cone}}
\newcommand{\Fib}{\operatorname{Fib}}

\newcommand{\Gm}{{\mathbb{G}_m}}
\newcommand{\GL}{\mathrm{GL}}
\newcommand{\SH}{\mathbf{SH}}

\newcommand{\Spec}{\operatorname{Spec}}
\newcommand{\Supp}{\operatorname{Supp}}
\newcommand{\Ker}{\operatorname{Ker}}
\newcommand{\Coker}{\operatorname{Coker}}

\newcommand{\chark}{\operatorname{char}}
\newcommand{\pr}{\operatorname{pr}}

\newcommand{\G}{\mathbf G}
\newcommand{\SHd}{\mathbf{SH}}
\newcommand{\dSH}{\mathbf{SH}}

\newcommand{\pour}{\ar@{}[ur]|(0.2){\text{\pigpenfont C}}}
\newcommand{\podr}{\ar@{}[dr]|(0.2){\text{\pigpenfont I}}}

\newcommand{\AnuFull}{(\A^1-\{0,1\})\times S}

\newcommand{\AnupFulls}{ (\A^1-\{0,1\})_+}

\newcommand{\Anu}{(\A^1-\{0,1\})}

\author{A.~Druzhinin}
\newcommand{\Addresses}{{
  \bigskip
  \footnotesize
  
  Druzhinin. A., \textsc{Chebyshev Laboratory, St. Petersburg State University, 14th Line V.O., 29B, Saint Petersburg 199178 Russia.}\par\nopagebreak
  \textit{E-mail address},  Druzhinin. A.\texttt{andrei.druzhGgmail.com}
}}

\newcommand{\unKMW}{\underline{\mathrm{K}}^\mathrm{MW}}

\begin{document}

\title{The homomorphism of presheaves ${\mathrm{K}}^\mathrm{MW}_*\to {\pi}^{*,*}_s$ over a base.}
\thanks{Research is supported by «Native towns», a social investment program of PJSC «Gazprom Neft».}

\maketitle
\begin{abstract}
We construct the homomorphism of  presheaves
${\mathrm{K}}^\mathrm{MW}_*
\to {\pi}^{*,*}$ over an arbitrary base scheme $S$,
where $\mathrm{K}^\mathrm{MW}$ is the (naive) Milnor-Witt K-theory presheave. 

Also we discuss some partly alternative proof (or proofs) of the isomorphism of sheaves $\unKMW_n\simeq \underline{\pi}^{n,n}_s$, $n\in \mathbb Z$, over a filed $k$ originally proved in \cite{M02} and \cite{M-A1Top}.
\end{abstract}


\section{Introduction}

The presheaf of the (naive) Milnor-Witt K-theory $\mathrm{K^{MW}_*}$ is defined as a graded ring with generators $[a]\in \mathrm{K^{MW}_{1}}$, $\forall a\in \mathbf G_m$ and $\eta\in \mathrm{K^{MW}_{-1}}$ and relations 
\begin{equation}\label{eq:KMWrelations}\begin{array}{rrcll}
\text{(Steiberg relation)}& [x][1-x] &=& 0, &\forall x\in (\G_m-\{1\}),\\
& \eta[x][y]&=&[xy]-[x]-[y], &\forall x,y\in \G_m\\
& [x]\eta &=& \eta[x], &\forall x\in \G_m\\
& \eta(\eta[-1]+2)&=&0.
\end{array}\end{equation}
As shown in \cite[section 4.2.1]{DF_MW-Complexes} 
the result \cite[theorem 6.3]{GSZ-MWloc} implies that the Zariski sheafification 
$\underline{\mathrm{K}}^{\mathrm{MW}}_*$ 
of the presheaf $\mathrm{K^{MW}_*}$ over an infinite filed $k$ of odd characteristic 
is equal to the unramified Milnor-Witt K-theory sheaf $\mathbf{K}^{\mathrm{MW}}_*$ defined in \cite[section 3]{M-A1Top}, which is by defined as 
an unramified sheaf 
that is equal to the (naive) $\mathrm{K}^\mathrm{MW}$ on fields.
The stable version of the Morel's theorem \cite[theorem 19, cor. 21]{M-A1Top} states isomorphism of 
sheaves $\mathbf{K}^{\mathrm{MW}}_*\simeq \underline{\pi}_s^{*,*}$
for a (perfect) base filed $k$ of an arbitrary characteristic.

The result of the paper is the following 
\begin{theorem}
The assignment 
\begin{equation}\label{eq:assign}\begin{array}{lclcll}
[x]&\mapsto& [pt\mapsto x]&\in& [pt,\G_m^{\wedge\phantom{-} 1}]_{\SH(S)}\\
\eta&\mapsto& \Sigma^{-2}_{\G}[m-p_1-p_2]&\in& [pt,\G_m^{\wedge -1}]_{\SH(S)}, 
\end{array}\end{equation}
where $m\colon \G_m^2\to \G_m\colon (x,y)\mapsto xy$, and $p_1,p_2\colon \G^2_m\to \G_m$ are the projections,
induces 
the homomorphism of presheaves ${\mathrm{K}}^{\mathrm{MW}}_*\to {\pi}^{*,*}$ 
for any base scheme $S$.
\end{theorem}



Since there is a canonical endomorphism $\mathrm{K}^\mathrm{MW}_{*}(S)\to \mathrm{K}^\mathrm{MW}_{*+1}(S\wedge \G_m)$ given by $\phi\mapsto [t]\phi$, and $[\G_m\wedge X, \G_m\wedge Y]=[X,Y]$ in $\SH(S)$,
the theorem follows directly form  
\begin{proposition}
The following equalities hold in the stable motivic homotopy category $\dSH(S)$ for all $S$
\begin{equation}\label{eq:SHK^MWrelations}\begin{array}{lclcl}
[(x,1-x)] &=& 0&
\in  &[\AnupFulls,\G_m^{ \wedge 2 } ]_{\dSH(S)},\\ 
\Sigma\eta[(x)][(y)] &=& \Sigma^2_\G([m]-[{p_1}]-[{p_2}])&
\in &[\G_m^{\wedge 2}\wedge(\G_m^2)_+, \G_m^{\wedge 3}]_{\dSH(S)},\\
\empty [(x)]\Sigma\eta &=&\Sigma\eta [(x)]&
\in &[\G_m^{\wedge 2}\wedge (\G_m^1)_+, \G_m^{\wedge 2}]_{\dSH(S)} \\
(\Sigma\eta)^2[-1] + 2\Sigma\eta &=& 0&
\in &[\G_m^{\wedge 2},pt_+]_{\dSH(S)}
\end{array}\end{equation}
where 
the products in equalities are the external product with respect to the monoidal structure, and
\begin{itemize}[leftmargin=9pt]
\item[-]
$\Sigma\eta$ denotes $\Sigma^2_{\G}\eta$,
\item[-]
$(1-x,x)\colon \AnuFull\to \G_m\times\G_m$ is a regular map, and $[(1-x,x)]$ is its class in $\dSH(S)$,
\item[-]
$(x),(y)\colon \G_m\to \G_m$ denote two copies on the identity map, $[(x)],[(y)]\in [(\G_m)_+, \G_m^{\wedge 1}]$ denote their classes, and
$[-1]\colon [pt_+,\G_m^{\wedge 1}]$ denotes the class of $ pt\to \G_m$ given by $-1$,
\item[-] $m, p_1,p_2\colon \G_m^2\to \G_m$ are the product map and the projection maps, and 
$[m],[p_1],[p_2]\colon [\G_m^{\wedge 2}, \G_m^{\wedge 1}]$ are the induced homomorphisms.
\end{itemize}
\end{proposition}


We should say that the Steinberg relation in $\SH$ (the first one in the list above) 
is proven originally by Po Hu and Igor Kriz in \cite{PHandIK-Staiberg}. 
The proof is written for the base filed base but it works as well for an arbitrary base.
To keep the text being complete
we present here the short alternative argument.
From what the author understands this argument is essentially equivalent to the original proof.

\subsection{Strategy of the proof}

\subsubsection{Steinberg relation.}

The Steinberg relation follows from that the 
class of a morphism $c\colon U\to \G_m\times\G_m$ in the group $[U_+,  \G_m^{\wedge 2}]_{\SHd(S)}$
is equal to a composition $U_+\to pt_+\to \G_m^{\wedge 2}$,
if $c$ fits into a diagram $$\xymatrix{
U\ar[d]\ar[r]& X\ar[d]&  X\times_{\A^2} ((\A^1\times\{1\})\cup (\{1\}\times \A^1))\ar[l]\ar[d]\ar[rd]&\\
\G_m^2\ar[r]& \A^2 &(\A^1\times\{1\})\cup (\{1\}\times \A^1)\ar[l] & \{(1,0),(0,1)\}\ar[l] 
,}$$
such that $X\simeq pt\in \SH(S)$.
Then applying this to $U=\Anu\amalg \G_m=Z(x+y-1)\amalg Z(x-1)\subset \G_m^2$
we see that the classes of the maps $\Anu_+\to \G_m^{\wedge 2}\colon (t)\mapsto (t,1-t)$, and $(\G_m)_+\to \G_m^{\wedge 2}\colon (t)\mapsto (1,t)$ both are 
equal to the same constant. But the class of the second one is trivial, hence the class of the first one too.

In non-stable case our the first argument requires the stabilisation by one $S^1$ suspension. 

\subsubsection{Other relations}
The image of the homomorphism $\mathrm{K}^\mathrm{MW}_n\to \pi^{n,n}(S)$ are the sum of class $\Sigma^{-l}_\G f$ for a regular morphisms $f\in \G_m^{l_1}\to \G_m^{l+n}$.
Let $\bullet$ denotes the external product (composition) of morphisms with respect to the monoidal structure, 
and $\circ$ denotes the composition of morphisms in the categorical sense.
The last three relations from \eqref{eq:SHK^MWrelations} follows from the following observations:
\begin{itemize}[leftmargin=50pt]
\item[(prop \ref{prop:composashort})] 
$\Sigma^{r_1}_\G[f\circ g]= \langle -1\rangle^{m  n} f \bullet g =  \Sigma^{r_2}_\G\langle -1\rangle^{m(m^\prime+1)} (g \circ \Sigma^{m-m^\prime}_\G f)$\\ 
for any $f \in Map(\G_m^{n},\G_m^{m^\prime})$, $g\in Map(\G_m^{ m}, \G_m^{ n})$ for some $r_1,r_2\in \mathbb Z$;
\item[(lm \ref{lm:Twist})]
$[T]=\Sigma^2_\G\langle -1\rangle\in [\G_m\wedge\G_m,\G_m\wedge\G_m]_{\SH(S)}$, \\where $T$ is a permutation on $\G_m^2$ and $\Sigma^1_\G\langle -1\rangle\in [\G^{\wedge 1}_m,\G^{\wedge 1}_m]_{\SH(S)}$ 
is the class of the map $\G_m\to \G_m\colon t\mapsto -t$.
\item[(rem \ref{rm:Tm})]
$m\circ T=m$, 
where $m\colon \G_m^2\to \G_m$ is the product.
\end{itemize}
The first equality is a variant of the fact that two groupoid operations satisfying the property $(f_1\circ f_2)\bullet (g_1\circ g_2)=(f_1\bullet g_1)\circ (f_2\bullet g_2)$ are equal and commutative,
Actually the first equality 
follows from the second one from the list above. 
The second equality can be proven either as a consequence of the fact that elementary transformation (over $\mathbb Z$) which permutes coordinates.
In the nonstable case the argument with elementary transformations
uses stabilisation by one $S^2$.
So the proof of other three relations holds after the smash with $S^2$.

Note that alternatively 
the equality $T=\Sigma_{\G}\langle -1\rangle$
can be obtained using the framed permutation homotopy on $\G_m^{\wedge 2}$ form \cite{AGP-CancTh}), but this argument requires $\P^1$ stabilisation.
%
%
\subsection{Proofs for the case of a base field}
\subsubsection{(The universal strongly homotopy invariant theory)}
In the Morel's book \cite{M-A1Top} the homomorphism $\mathrm{\mathbf{K}}^\mathrm{WM}_n\to \pi^{n,n}$, where $\mathrm{\mathbf{K}}^\mathrm{WM}_n$ is the sheaf of the unramified Milnor-Witt K-theory,  
follows from the universal property of the sheaf $\mathrm{\mathbf{K}}^\mathrm{WM}_n$ in the class of 
strongly homotopy invariant sheaves over a filed. 
In the case of non-zero dimensional base it is unknown does the sheaves $\pi^{n,n}$ are strongly homotopy invariant sheaves. 
In the same time
some key inner arguments from \cite{M-A1Top} in the proof of the last three relations \eqref{eq:KMWrelations}
looks being general and should work over an arbitrary base.
So the author doesn't know entirely is it possible to prove this relations
relations in $\pi^{n,n}$ over any $S$ using the arguments from \cite{M-A1Top}. 

\subsubsection{(The Steinberg relation)}
The Steinberg relation in $\SH(k)$ for an arbitrary filed $k$ was proven originally 
by Po Hu and Igor Kriz in \cite{PHandIK-Staiberg}, and reproven by Geoffrey M.L. Powell in \cite{Powell-Steinberg}.
The arguments of both proofs can be word by word repeated in the case of an arbitrary base. 

The author apologize for the doubts in the correctness of the arguments \cite{PHandIK-Staiberg}, which he had wroted in previous version, now he had understend the original proof. 
\par Nevertheless the author still do not understand the alternative proof in \cite{Powell-Steinberg}. In \cite[definition 3.0.7, proposition 3.0.8]{Powell-Steinberg}  nothing is mentioned about the base point in the $\A^n$ for the morphism $X\to \A^n$, 
and is the cone considered in \cite[proposition 3.0.8]{Powell-Steinberg} is just the cone of the morphism the unpointed varieties  $X\to \A^n$ then in $\mathbf {SH}$ it is equivalent to $S^1\wedge \Fib_{\SHd}(X_+\to pt_+)$, but not to to the suspension of $X_+$.
The author would appreciate if some one can explain what is meant there. 
%
%
%

\subsubsection{(Framed correspondences)}
Alternatively the relations of Milnor-Witt K-theory in the stable motivic homotopy groups over fields ware proven in  \cite{Nesh-FrKMW} by A.~Neshitov. 
The prove is given by precise framed homotopies and 
the relations  are proven in  
$H^0(\mathbf ZF(\Delta^\bullet,\G_n^{\wedge l}))$, which is formally stronger then relations in $\pi^{n,n}(k)$.
In the same time the proof of Steinberg relation requires assumption that the base field is of characteristic different form $2$ and $3$.
From what the author understands at least some of the framed homotopies used in the proof could be lifted at least to henselian local bases.
If this is true for all homotopies, then this would imply 
the homomorphism of Nisnevich sheaves $\underline{\mathrm{K}}^\mathrm{MW}_n\to \underline{\pi}^{n,n}$ over $\mathbf{Z}[1/6]$.

\subsection{Acknowledgement} The author is grateful to the J.~
I.~
Kylling for the discussions and comparing of the proofs of the Steinberg relation given in  \cite{PHandIK-Staiberg}, \cite{Powell-Steinberg}, and the present one.
The author is grateful for the participants of the Chebyshev laboratory seminar on K-theory and motivic homotopy theory and especially for A.~
Ananievsky for the corrections and useful comments on the text.
The author thanks M.~
Hoyois for the explanation and the reference for the construction on $\SHd(S)$ over an arbitrary scheme $S$. 

\subsection{Notation:}
All products, points, and schemes are considered relatively over the base scheme $S$.

\section{Proof of the Steinberg relation}
\subsection{The reduced curve.}

In the subsection we prove the Steinberg relation up to some constant, i.e we prove that the morphism 
$(1-x,x)\colon \Anu\to \G_m^2$ 
can be passed in $\SHd(S)$ 
throw $(\Anu)_+\to pt_+\to \G_m^2$ in $\SHd(S)$. 
We refer reader to \cite[Appendix C]{Hoyois2013} for the definition of the stable motivic homotopy category $\mathbf{SH}(S)$ over an arbitrary scheme $S$.

\begin{notation}
For any $X\in Sm_S$ denote by $X/pt$ the fibre of the morphism $X\to pt$ in $\SHd(S)$ $$X/pt\to X_+\to pt_+\to (C/pt)\wedge S^1.$$ So any morphism $X\to Y$ for $X,Y\in \SHd(S)$ induces the morphism $C/pt\to Y$ via the composition $C/pt\to C\to Y$.
\end{notation}


\begin{lemma}\label{sbl:CurveinSH}
Let $X$ be a scheme over $S$, 
and assume that $X$ is $\A^1$ contractable, i.e. the canonical morphism $X\to pt$ is equivalence in $SH(S)$.
Let $\xi\colon X\to \A^2$ be a morphism of schemes such that $X\times_{\A^2} (Z(xy)-Z(x+y-1))=\emptyset$, where $x$ and $y$ denotes coordinate functions on $\A^2$ (and so $Z(xy)-Z(x+y-1)\simeq (\A^1-\{1\})\coprod_{\{0\}}(\A^1-\{1\})$).

Denote $U=X\times_{\A^2} \G_m^2$ and $c\colon U\to \G_m^2$.
Then the class of the morphism $c$ in $[U/pt,\G_m{\wedge}\G_m]$ is trivial.
\end{lemma}
\begin{proof}
\newcommand{\Gmp}{\G_{m,+}}
\newcommand{\Gmd}{(\G_m,1)}
\newcommand{\Aod}{(\A^1,1)}
It follows form the assumption on $\xi$ that $X\times_{\A^2} \{(0,0)\}=\emptyset$.
Consider the diagram of the triangles in $\SHd(S)$
\begin{equation}\label{eqdiag:StRelEq}\xymatrix{
\left(
\G_m\times \{1\}\cup \atop \{1\}\times \G_m
\right)_+
\ar[r]\ar[d]& 
\left( 
(\A^1\times\{0\}) \cup\atop
(\{0\}\times \A^1)
\right)_+
\ar[d]^{\alpha} \ar[r]& 
T^{\wedge 4}\ar[d]\\
\Gmp\vee\Gmp\ar[r]^{\beta}\ar[d]& 
(\A^2-0)_+ \ar[r]\ar[d]& 
\Cone(\beta)\ar[d]\\
\G_m\wedge\G_m \ar[r]^{\gamma}& 
\Cone(\alpha)\ar[r]& \Cone(\gamma)\ar[r]&
\G_m^{\wedge 2}\wedge S^1\\
U_+\ar[r]\ar[u]& 
X_+ \ar[r]\ar[u]& 
X/U\ar[u]^{\delta}
\ar[r]&U_+\wedge S^1\ar[u]\\
U/pt\ar[r]^{\zeta}\ar[u]& X/pt \ar[r]\ar[u]& Cone(\zeta)\ar[u]^{\simeq}\ar[r]^{\simeq}&U/pt\wedge S^1\ar[u]
}\end{equation}
where $X/U = \Cone(U\to X) = \Cone(U_+\to X_+)$.

It follows form the assumption $\xi$ 
that 
$X/U\simeq X/U_1 \vee X/U_2\simeq (X-Z_1)/(U-Z_1)\vee (X-Z_2)/(U-Z_2)$,
where 
$Z_1 = (X\times_{\A^2} \A^1\times \{0\})_{red} = X\times_{\A^2} \{(1,0)\}$, 
$Z_2 = X\times_{\A^1} \{0\}\times \A^1 = X\times_{\A^2} \{(0,1)\}$,
and $U_1= X-Z_2$, $U_2 =X-Z_1$.

Note that 
$$\Cone(\beta)\simeq (\Gmd\wedge T)\vee (T\wedge\Gmd)\simeq \big((\A^2-0)/(\G_m\times\A^1)\big)\vee\big((\A^2-0)/(\A^1\times \G_m)\big).$$
Let $\xi_1, \xi_2\to X\to \A^1$, $\xi=(\xi_1,\xi_2)$.
The homotopy
$$
(X-Z_1)/(U-Z_1)\times \A^1\to (\A^2-0)/(\G_m\times\A^1)\colon p\mapsto (\xi_1(p),(1-\lambda)\xi_2(p) + \lambda), 
$$
implies that the morphism $(X-Z_1)/(U-Z_1)\to \Cone(\beta)$ induced by $\xi$ in $\SH(S)$ is trivial.
Similarly the morphism $(X-Z_1)/(U-Z_1)\to \Cone(\beta)$ is trivial.
Hence the vertical arrow $X/U\to \Cone(\beta)$ in the diagram is trivial.

Since $X\to pt$ is an isomorphism by assumption, it follows that the last arrow in the last row in the diagram is isomorphism;
then since the second last vertical arrow in the diagram is isomorphism, 
it follows that the composition $S^1\wedge U/pt\to S^1\wedge (U_+)\simeq \G_m\wedge \G_m$ is trivial. The claim follows.
\end{proof}

\begin{proposition}\label{prop:StRel}
The class of the morphism $\Anu\to \G_m\colon (t)\mapsto (t,1-t)$ in $[\Anu, \G_m^2]_{\SH(S)}$ is trivial.
\end{proposition}
\begin{proof}
Applying lemma \ref{sbl:CurveinSH} to the closed subscheme $X=Z((x-1)(x+y-1))\subset\A^2$
we see that the morphism $c\amalg c^\prime\colon \Anu \amalg \G_m\to \G_m^2$ in $[\big(\Anu \amalg \G_m\big)/pt, \G_m^{\wedge 2}]$, where $c\colon \Anu\to \G_m^2\colon (t)\mapsto (t,1-t)$, and $c^\prime\colon \G_m\to \G_m^2\colon (t)\mapsto (t,1)$, is trivial.
Hence the class $[c\amalg c^\prime]$ of the morphism $c\amalg c^\prime$ is $\SH(S)$ can be passed throw 
$$\big(\Anu \amalg \G_m\big)_+\to pt_+ \to \G_m^2.$$
Now since the class of the composition $c^\prime \circ 1$, where $1\colon pt_+\to (\G_m)_+$ is given by the point $\{1\}$, defines the zero morphism in the group $[pt_+,  \G_m^2]_{\SH(S)}$, 
it follows that $[c\amalg c^\prime]=0\in [\Anu \amalg \G_m, \G_m^2]_{\SH(S)}$.
Hence the class of the morphism $\Anu\to \G_m\colon (t)\mapsto (t,1-t)$ in $[\Anu, \G_m^2]_{\SH(S)}$ is trivial.
\end{proof}

\section{Proof of other relations of Milnor-Witt K-theory}\label{sect:therrrel}

Denote $\SH=\SH(S)$, $Sm=Sm_S$.

\begin{definition}
Let 
$f\in Map(X, Y)$, $g\in Map(X^\prime, Y^\prime)$, for $X,Y,Z\in Sm_S$,  (or
$f\in [X, Y]_{\SH}$, $g\in [X^\prime, Y^\prime]_{\SH}$, for $X,Y,Z\in \SH(S)$,).
\item[$\circ\;\colon $]
Then if $Y=X^\prime$, we can define the composition morphism in $Map(X,Y^\prime)$ (or $[X,Y^\prime]_\SH$) which we denote by $g\circ f$.
\item[$\bullet\;\colon $]
Denote by $f\bullet g\in Map(X\times X^\prime, Y\times Y^\prime)$ 
the (external) product, which we also call as an external composition, the same notation we use for
$f\bullet g\in Hom_{\mathbb ZSm^{Id}}(X\times X^\prime, Y\times Y^\prime)$ or $ [X\wedge X^\prime,Y\wedge Y^\prime]_\SH$;
let us note that we use both notations $f g$ and $f\bullet g$ for the external products in $SH$, but only $f\bullet g$ for $Sm$ and $\mathbb ZSm^{Id}$;
\item[$\Sigma_{\G_m}\colon$] denote $\Sigma^l_{\G} f = id_{\G^l} \bullet f$, $f \Sigma^l_{\G} = f\bullet id_{\G^l}$. 
\item[$\sim^{\G}\colon$] let us write $f\sim^{\G} g$ iff $\Sigma_\G^{l} f = \Sigma_\G^{l^\prime} g$ for some $l, l^\prime\in \mathbb Z$.
\end{definition}
\begin{remark}
For any 
$f\in Map(\G_m^{n}, \G_m^{m^\prime})$
$g\in Map(\G_m^{n^\prime}, \G_m^{m})$,
$ f\Sigma^m_{\G} \circ \Sigma^n_\G g  =  f \bullet g \in Map(\G_m^{n+n^\prime}, \G_m^{m^\prime+m}) $
%
\end{remark}
\begin{remark}
If $f\sim^\G f^\prime$, $g\sim^\G g^\prime$, $f^\prime=g^\prime$ then $f=g$.
\end{remark}
\begin{definition}
Define regular maps 
\item[-] $m\colon \G_m\times\G_m\to \G_m\colon (x,y)\mapsto xy$,
\item[-] $(a)\colon pt\to \G_m\colon pt\mapsto a$,
\item[-] $m_a \colon \G_m\to \G_m\colon x\mapsto ax$.
\par\noindent Define 
\item[-] 
$\eta\in [pt,\G_m^{\wedge -2}]_{\SH}$, $\eta = \Sigma^{-2}_{\G}(m-p_1-p_2)$, $p_1,p_2\colon \G_m\times \G_m\to \G_m$ be the projections, 
\\ 
denote $\Sigma\eta= \Sigma^{2}_\G \eta$, note that $\Sigma\eta$ here is just a symbol;
\item[-] $[a]\in  [pt, \G_m^{\wedge 1}]_{\SH}\colon pt\mapsto a$,
\item[-] $\langle a\rangle \in [pt,pt]_{\SH}$, $\langle a\rangle = \Sigma_\G^{-1}(m^a - (a))$. 
\end{definition}

\begin{definition}
Let $\G_m\simeq \G_m^{\wedge 1} \oplus pt\in \SH$ be the isomorphism  given by the point $1\in \G_m$.
For a regular map $f\in Map(\G_m^{n}, \G_m^{m})$
denote by $\overline f\in 
Hom_{\mathbb ZSm^{Id}}(\G_m^{\wedge n}, \G_m^{\wedge m})$ 
the induced morphism in the Karoubi envelope of the linearisation of $Sm$. 

For any morphism $f\in \mathbb ZSm^{Id}$, $f\in Hom(X,Y)$ denote by $[f]\in [X,Y]_{\SH}$ the class of the morphism in $\SH$. 
\end{definition}
\begin{example}
Let $m\colon \G_m\times\G_m\to \G_m\colon (a,b)\mapsto ab$ be the multiplication morphism, 
\item[-] $[\overline m] = \Sigma^2_{\G}\eta\in [\G_m^{\wedge 2}, \G_m^{\wedge 1}]$, 
\item[-] $[\overline {(a)}] = [a]\in  [pt, \G_m^{\wedge 1}]_{\SH}$;
\item[-] $[\overline{ m_a}] = \Sigma_\G^{1}\langle a\rangle \in  [\G_m^{\wedge 1}, \G_m^{\wedge 1}]_{\SH}$.
\end{example}
\begin{remark}\label{rm:computredcom}
Computing the composition of morphisms $\overline{g}\circ\overline{f}$,
$f\in Map(\G_m^{n}, \G_m^{m})$, $g\in Map(\G_m^{m}, \G_m^{l})$, 
it is suitable to think about the morphisms induced 
by $\overline{f}$ (and $\overline{g}$) in $[\G_m^{n}, \G_m^{m}]_{\SHd}$, 
which is given by the formula $P_n \circ f \circ P_m$ where $P_n=\prod_{i=1\dots n} (id_{\G_m^n} - 1\circ p_i)$, $1\colon pt\to \G_m$, $p_i\colon \G_m^n\to \G_m^{n-1}$ is the projection along the $i$-th multiplicand. 
\end{remark}

\begin{definition}
Let
\item[-] $T$ be the permutation on $\G_m^2$;
\item[-] $H=id_{\G} +m_{-1}\in \mathbb Z{Sm}(\G_m,\G_m)$, 
\item[-] $h=\Sigma^{-1}_\G[\overline H]=\langle 1\rangle +\langle -1\rangle\in [pt,pt]_{\SH}$.
\end{definition}
\begin{remark}\label{rm:Tm}
By commutativity we have $m\circ T=m$.
\end{remark}
\begin{lemma}
\label{lm:Twist}
For a permutation $P\in Aut_{Sm}(\G_m^n)$ with the sign $s$
$$[\overline{P}] = \Sigma_\G^{N}\langle -1\rangle^s=\Sigma_\G^{1}\langle -1\rangle^s\Sigma_\G^{N-1}.$$ 
\end{lemma}
\begin{proof}
Since $\G_m^{\wedge l}\wedge S^l = T^{\wedge l}=\A^l/(\A^l-0)$ the claim follows from the fact that any permutation defines the matrix in the subgroup in $\GL(\mathbb Z)$ generated by elementary matrices and the matrix the diagonal matrix $(-1,1,\dots 1)$. Let us note in addition that the general case follows from the case of the twist on $\G^2_m$. 
 linear homomorphism is equal to  
\end{proof}

\newcommand{\sign}{\operatorname{sign}}

\begin{proposition}\label{prop:composashort} 
For any $f \in [\G_m^{\wedge n},\G_m^{\wedge m^\prime}]_{\SH}$, $g\in [\G_m^{\wedge m}, \G_m^{\wedge n}]_{\SH}$,
we have 
$$f\circ g\sim^\G \langle -1\rangle^{m  n} f \bullet g
\sim^\G  \langle -1\rangle^{m(m^\prime+1)} (g \circ \Sigma^{m-m^\prime}_\G f)$$
\end{proposition}
\begin{proof}
The first equivalence follows form
\begin{multline*}
f\circ g \sim^\G \Sigma^{1+n}_\G f\circ \Sigma^{1+n}_\G g = \\ \Sigma^n_\G f\circ \check  P\circ \Sigma^n_\G g \circ \hat P =
\Sigma^n_\G f\circ (\langle -1\rangle^{n(n+1)}\Sigma_\G^{2n})\circ \Sigma^n_\G g \circ (\langle -1\rangle^{n m}\Sigma_\G^{2n})=\\
\langle -1\rangle^{n m }\bullet f \bullet g
\end{multline*}
where $\check P\colon \G^{\wedge n}_m\wedge \G^{\wedge n+1}_m\to\G^{\wedge n+1}_m\wedge \G^{\wedge n}$,
and $\hat P\colon \G^{\wedge n}_m\wedge \G^{\wedge n+1}_m\to\G^{\wedge n+1}_m\wedge \G^{\wedge n}$. 
Note that the sign of the permutation $(1,\dots l,l+1\dots l+k)\to (l+1,\dots l+k, 1\dots l)$ is equal to 
$l(l+k+1)=l^2+l(k+1)=l(1+k+1)=lk(mod 2)$, $\forall l,k\in \mathbb Z_{\geq 0}$.
The second equivalence follows form the first one applied to $\Sigma^{m-m^\prime}_\G f$ and $g$.
\end{proof}
\begin{remark}
The sign in \ref{cor:composa}(2) is the only one sign which we essentially use in the proof of relation of Milnor-Witt K-theory in $\SH$.
\end{remark}
The first relation in the list \eqref{eq:SHK^MWrelations} follows almost tautologically form proposition \ref{prop:composa} and the definition of $\eta$.
\begin{lemma}\label{lm:ProdRel} The following equality for morphisms in $\SH$ holds
$$\Sigma\eta\bullet[a]\bullet[b] = [p\circ (m-\overline{p_1}-\overline{p_2})]\in [\G_m\wedge\G_m, \G_m^{\wedge 1}]_{\SH},$$ where $p_1,p_2\colon \G_m^2\to \G_m$ are projections, and 
where $a,b,p\colon \G_m\to \G_m^{\wedge 1}$ denotes three copies of the canonical projection. 
\end{lemma}
\begin{proof}
It follows from prop \ref{prop:composa}(1) that $$
\Sigma\eta\bullet[a]\bullet[b]
\stackrel{cor \ref{cor:composa}(2)}=
\Sigma\eta\circ([a]\bullet[b])=[\overline{m}\circ \overline{((a)\bullet (b))}]=p\circ (m-\overline{p_1}-\overline{p_2})
$$
where $\Sigma\eta = \Sigma_\G^2\eta$.
\end{proof}

Now we prove the other two relation.

\begin{proposition}\label{prop:CommRel}
The following equality holds
$$[a]\Sigma\eta =\Sigma\eta[a]\in [\G_m^{\wedge 2}, \G_m^{\wedge 2}]_{\SH}.$$ 
\end{proposition}
\begin{proof} 
\[
[a] \eta =
[(\overline{a})\bullet \overline{m}] \stackrel{}{=}
[\overline{T}\circ(\overline{m}\bullet(\overline{a}))]\stackrel{prop \ref{prop:composashort}
}{=}
[(\overline{m}\bullet\overline{(a)})\circ \overline{T} ]\stackrel{}{=}
[(\overline{m}\circ \overline{T})\bullet\overline{(a)}]\stackrel{rem\ref{rm:Tm}}{=}
[\overline{m}\bullet\overline{(a)}]=
\eta [a]
\]
\end{proof}
\begin{proposition}\label{prop:quad}
The following equality holds
$$\eta^2[-1] + 2\eta = 0\in [pt,\G_m^{\wedge -2}]_{\SH}.$$
\end{proposition}\begin{proof}
Recall $\Sigma\eta = \Sigma_\G^2\eta$. 
Using prop \ref{prop:composashort} 
we have 
$$\Sigma\eta (\Sigma\eta[-1])=
[\overline{m}\bullet (\overline{m}\bullet\overline{(-1)})]
\stackrel{cor \ref{cor:composa}(2)}{=}
[\overline{m}\circ (\overline{m}\bullet\overline{(-1)})],$$
and 
the straightforward computation in the Karoubi envelope of the linearisation of $Sm$ in view of rem \ref{rm:computredcom} shows that
$$\overline{m}\circ (\overline{m}\bullet[-1])\colon \G_m^2\to \G_m\colon (x,y)\mapsto (-xy)-(-x)-(-y)-(xy)+(x)+(y)-(-1).$$
So $\eta\bullet (\eta\bullet[-1]+2id_{\G_m^{\wedge 2}})=\Sigma^{-2}_{\G}[\overline{m\circ (m\bullet (-1) )+2m}]$, and
\begin{multline*}
\overline{m}\circ (\overline{m}\bullet \overline{(-1)} )+2\overline{m}\colon \G_m^2\to \G_m\colon \\
(x,y)\mapsto (-xy)+(xy)-((-x)+(x))-((-y)+(y)) -((-1)+(1))=\\ \overline{H}((xy)-(x)-(y)+(1))=\overline{H}(\overline{m}(x,y)).
\end{multline*}
Thus we have got
$$\eta\bullet\eta\bullet[-1]+2\eta = \Sigma^{-2}_\G[\overline{H}\circ \overline{m}] \; (\,\stackrel{prop \ref{prop:composashort} }{=} h \eta\,).$$
Now 
we see 
$$
[\overline{H}\circ \overline{m}]
\stackrel{prop\ref{prop:composashort}
}{=} [\overline{m}\circ \Sigma_{\G_m}^1\overline{H}] 
\stackrel{lm\ref{lm:Twist}}{=} [\overline{m}\circ \overline{(id_{\G_m^2} - T)}] \stackrel{rem\ref{rm:Tm}}{=} 0$$
\end{proof}

\section{The homomorphism $K^{MW}_n(S)\to H^0(\mathbb ZF(\Delta_S,\G_m^n))$.}

In the section we lift the homomorphism 
$K^{MW}_n(S)\to \pi^{n,n}_S$ to the level of framed correspondences.

Let us briefly recall definition of the category $\SH^{fr}(S)$, see \cite{ElHoKhSoYa-MotDeloop}.
Consider the $\infty$-category of additive presheaves of $S^1$-spectra with framed transfers $Pre^{\Sigma}(Corr^{fr}_S)$ over $S$.
Let $\SH^{fr}_{\A^1,S^1}=Pre^{\Sigma}_{\A^1}(Corr^{fr}_S)$ denotes the localisation with respect to morphisms $\A^1\to X$.
Let $\SH^{fr}_{\A^1}$ be the stabilisation of $\SH^{fr}_{\A^1,S^1}$ with respect to $\G_m$.
Then it follows form the usual (simplicial or topological) Hurevich isomorphism that $[pt, Y]_{\SH^{pre}_{\A^1}}=H^0(\mathbb ZF(\Delta_S,Y)$.


Since any regular map gives us a framed correspondences and since by the Cancellation theorem \cite{AGP-CancTh}
we have $[X\wedge \mathbb G_m, Y\wedge \mathbb G_m]_{\SH^{fr}(S)}$
we can consider the right side of the assignment \eqref{eq:assign}
$$\begin{array}{lclcll}
[x]&\mapsto& [pt\mapsto x]&\in& [pt,\G_m^{\wedge\phantom{-} 1}]_{\SH^{fr}(S)}\\
\eta&\mapsto& \Sigma^{-2}_{\G}[m-p_1-p_2]&\in& [pt,\G_m^{\wedge -1}]_{\SH^{fr}(S)}, 
\end{array}$$
where $m\colon \G_m^2\to \G_m\colon (x,y)\mapsto xy$, and $p_1,p_2\colon \G^2_m\to \G_m$ are the projections,
as morphisms in $\SH^{fr}(S)$.

\begin{proposition}
The similar assignment as \eqref{eq:assign} induces the homomorphism 
$$K^{MW}_n(S)\to H^0(\mathbb ZF(\Delta_S,\G_m^n)).$$
\end{proposition}
\begin{proof}
(\textbf{Steinberg relation}) In the proof of the Steinberg relation in $\SH(S)$ sublemma \ref{sbl:CurveinSH} we have essentially uses Zariski excision isomorphisms with respect to 
\begin{equation}\label{eq:ZarA1_A1}
(\A^1_{S\times\A^1},\A^1_{S\times\A^1}-0)\leftarrow(\A^1_{S\times\A^1}-Z,A^1_{S\times\A^1}-(0_S \amalg Z))
\end{equation}in the second two last rows of the diagram \eqref{eqdiag:StRelEq} applied to 
the morphism $\Anu\to \G_m\colon (t)\mapsto (t,1-t)$ as in prop \ref{prop:StRel}, and 
\begin{equation}\label{eq:ZarA2}(\A^2-0,\A^2-(1\times\A^1)) \leftarrow (\A^2-(\A^1\times 1), \A^2-(1\times\A^1\amalg \A^1\times 1)) \end{equation}  
in the second row.
Now let us see that lemma \ref{lm:ZarEx} yields that  \eqref{eq:ZarA1_A1} and \eqref{eq:ZarA2} are equivalences in $\SH^{fr}_{\A^1}$.
Actually the first case is immediate. In the case of \eqref{eq:ZarA2} it is enough to turn the picture 
and the consider the projection $\A^2\to \A^1\colon (x,y)\mapsto (x-y)$. Then \eqref{eq:ZarA2} becomes the particular case of the \eqref{eq:ZarA1_A1}.

(\textbf{Other relations}) Other relations in $\SH^{fr}$ follows by the same arguments as in section \ref{sect:therrrel}, all what we need that 
it follows form \cite{AGP-CancTh}
the permutation morphism on $\Gm^2$ is equal to $\langle -1\rangle$, $\A^1$-homotopy, where $\langle -1\rangle$ denotes the class of the framed corr. $(0,-t,pr)$ in $Fr_1(pt,pt)$. 
\end{proof}

\begin{lemma}\label{lm:ZarEx}

For any homotopy invariant presheaf $F$ over a base $S$ and closed subschemes $Z_1,Z_2\subset \A^1_S$ finite surjective over $S$, $Z_1\cap Z_2=\emptyset$ 
the canonical homomorphism 
$F(\A^1_S-(Z_1\cup Z_2))/F(\A^1_S-Z_2) = F(\A^1_S-Z_1)/F(\A^1_S)$

In other words 
the canonical homomorphism 
$i\colon 
\A^1_S-Z_1/\A^1_S-(Z_1\cup Z_2) \to
\A^1_S/\A^1_S-Z_2 
$ 
is an equivalence in $\SH^{fr}$,
where $\A^1_S-Z_1/\A^1_S-(Z_1\cup Z_2) $ and $\A^1_S/\A^1_S-Z_2 $ denotes the cones.
\end{lemma}
\begin{proof}
For any scheme $X$ over $S$
a function $\phi\in \mathcal O(\A^1_X)$ such that $Z(\phi)$ is finite over $X$. 
we can define a framed correspondence $(Z(\phi), \A^1_{X}, \phi, pr)\in Fr(X,\A^1_S)$, where $\pr\colon \A^1\times X\to \A^1_S$.
Then for a given section $s\in \Gamma(\A^1_X,\mathcal O(n))$ such that $s\big|_{\infty\times X}$ is invertible we can apply the construction to the function $s/t_\infty^n$. Denote the resulting correspondence by $\langle s\rangle$.

Moreover if $E\subset X$ $D_1,D_2\subset \A^1_S$ are a closed subschemes, and
$s\big|_{X\times_S D_1}$ and $s\big|_{(X-E)\times D_2}$ are invertible then the construction $\langle s\rangle$ gives us the correspondence between pairs, i.e. an element in $Fr((X,X-E), (\A^1-D_1,A^1-(D_1\cup D_2))$.

Let  $\delta\in\Gamma(\PP^1_{\A^1\times S},\mathcal O(1))$, $Z(\delta)$ is the diagonal in $\A^1_{\A^1\times S}$. Then 
in view of the described construction $\langle \delta\rangle$ is equal to the $\sigma$-suspension of the identity element 
in $Fr(\A^1_S,\A^1_S)$ and consequently the identity elements in $Fr_1((\A^1_S-Z_1,\A^1-(Z_1\cup Z_2)), (\A^1_S-Z_1,\A^1-(Z_1\cup Z_2)) )$ and $Fr_1((\A^1_S,\A^1_S-Z_2),(\A^1_S,\A^1_S-Z_2))$.

By Serre theorem for large enough $n$ we find a sections 
$$
s\in\Gamma(\PP^1_{\A^1\times S},\mathcal O(n)),
s^\prime\in\Gamma(\PP^1_{\A^1\times S},\mathcal O(n-1))
$$
such that
$$\begin{array}{ccc}
s\big|_{\infty\times S\times \A^1} = t_0^n\big|_{\infty\times S\times \A^1},& 
s\big|_{Z_1\times\A^1} = t_\infty^n\big|_{Z_1\times\A^1},& 
s\big|_{Z_2\times \A^1} = \delta t_\infty^{n-1}
\\
s^\prime\big|_{\infty\times S\times \A^1} = t_\infty^{n-1}\big|_{\infty\times S\times \A^1},&& 
s^\prime\big|_{Z_2\times \A^1} = t_\infty^{n-1}
.\end{array}$$ 
Then
$\langle s^\prime \rangle\in Fr_1((\A^1_S,\A^1_S-Z_2),(\A^1_S,\A^1-Z_2))$ is equal to zero, 
and
$$\langle s \rangle\in Fr_1((\A^1_S,\A^1_S-Z_2),(\A^1_S-Z_1,\A^1-(Z_1\cup Z_2)))$$ 
is a left inverse up to a suspension to the canonical morphism in $i\colon (\A^1_S-Z_1,\A^1-(Z_1\cup Z_2)))\to (\A^1_S,\A^1_S-Z_2)$, where the homotopy between $\sigma id_{(\A^1_S,\A^1_S-Z_2)} $ and $i \circ \langle s \rangle$ is given by $$\langle \alpha s + (1-\alpha)\delta s^\prime \rangle\in Fr_{1}((\A^1_S,\A^1_S-Z_2)\times\A^1, (\A^1_S,\A^1_S-Z_2)).$$

On other side by Serre theorem for a large enough $n$ we find a sections
$$
s\in\Gamma(\PP^1_{\A^1\times S},\mathcal O(n)),
s^\prime\in\Gamma(\PP^1_{(\A^1_S-Z_1)},\mathcal O(n-1))
$$
such that
$$\begin{array}{lll}
s\big|_{\infty\times S\times \A^1} = t_0^n\big|_{\infty\times S\times \A^1},& 
s\big|_{Z_1\times\A^1} = t_\infty^n\big|_{Z_1\times\A^1},& 
s\big|_{Z_2\times \A^1} = \delta t_\infty^{n-1}
\\
s^\prime\big|_{\infty\times (\A^1_S-Z_1)} = t_\infty^{n-1}\big|_{\infty\times S\times \A^1},&
s^\prime\big|_{Z_1\times_S (\A^1_S-Z_1)} = t_\infty^n \delta^{-1}\big|_{Z_1\times_S(\A^1_S-Z_1)},& 
s^\prime\big|_{Z_2\times_S (\A^1_S-Z_1)} = t_\infty^{n-1}
.\end{array}$$ 
Then
$\langle s^\prime \rangle\in Fr_1((\A^1_S-Z_1,\A^1_S-(Z_1\cup Z_2)),(\A^1_S-Z_1,\A^1_S-(Z_1\cup Z_2)))$ is equal to zero, 
and
$$\langle s \rangle\in Fr_1((\A^1_S,\A^1_S-Z_2),(\A^1_S-Z_1,\A^1-(Z_1\cup Z_2)))$$ 
is a right inverse up to a suspension to the canonical morphism in $i\colon (\A^1_S-Z_1,\A^1-(Z_1\cup Z_2)))\to (\A^1_S,\A^1_S-Z_2)$, where the homotopy between $\sigma id_{(\A^1_S-Z_1,\A^1_S-(Z_1\cup Z_2))} $ and $\langle s \rangle\circ i$ is given by 
$$\langle \alpha s\big|_{\PP^1_{\A^1_S-Z_1}} + (1-\alpha)\delta\big|_{\A^1_S-Z_1} s^\prime \rangle\in Fr_{1}((\A^1_S-Z_1,\A^1_S-(Z_1\cup Z_2))\times\A^1, (\A^1_S-Z_1,\A^1_S-(Z_1\cup Z_2)) ).$$ 
\end{proof}

%

\section{The isomorphisms  $\underline{\mathrm{K}}^\mathrm{MW}_n\to \underline\pi^{n,n}_s$ and $\underline{\mathrm{K}}^\mathrm{MW}_n\to \underline h^0(\mathbb ZF(\Delta\times -,\G_m^n))$, $\chark k\neq 2$.}

It is proven in \cite{M02} and \cite{M-A1Top} that there is a canonical isomorphism of sheaves 
\begin{equation}\label{eq:isok}\underline{\mathrm{K}}^\mathrm{MW}_n\to \underline\pi^{n,n}_s\end{equation} 
over a base filed $k$ for all $n\in \mathbb Z$.. 
Precisely the proofs are written for the case of a prefect field $k$ and as mentioned in the remark in \cite{M-A1Top} the result for the non-perfect filed follows by the general base change argument.

In the section we present the proof of the isomorphism based on the theory of framed motives \cite{GP_MFrAlgVar} and theory of Chow-Witt groups ( see the recent works  \cite{NFELD-MWCycMod}, \cite{CF} for the $\chark k=2$ case)

\newcommand{\D}{\mathbf{Sh}}
\newcommand{\Sh}{\mathrm{Sh}}

\subsection{The proof using the theory of framed motives and Chow-Witt groups}\label{subsection:FrCWKMWpi} 

The results of the Garkusha Panin theory of framed motives \cite{GP_MFrAlgVar} implies in particular that \cite{GP_MFrAlgVar} 
it is proven that $\underline{\pi}^{n,n}_s\to \underline h^0(\mathbb ZF(\Delta\times -,\G_m^n))$ for a perfect field $k$.
As shown in \cite{ElHoKhSoYa-MotDeloop} the ganeral base change argument like as above extends the result to the case of a base schemes $S$ that are essentially smooth (and even pro=smooth) over some perfect $k$.

Combining the methods of framed correspondences and homotopies with the 
theory of 
Chow-Witt groups \cite{BM_PullBack}, \cite{Fasel-ChWittRing} \cite{Fasel-GroupsdeCW}
it is proven in \cite{Nesh-FrKMW} that 
$\mathrm K^\mathrm{MW}(k)\simeq\pi^{n,n}_s(k) $ in the case $\chark k=0$
In \cite{DrKil-finfileds_AND_MorTh} the result is extended to the case of a perfect fields $\chark k\neq 2$.

Let us note that the Chow-Witt groups the are used to prove the injectivity of the map.

Here we improve the argument that the argument for the proof of the surjectivity of the map $\mathrm K^\mathrm{MW}(k)\simeq\pi^{n,n}_k$ for an arbitrary filed, actually it is done by the moving lemma proved in the next section.
Then we deduce the isomorphism \eqref{eq:isok} for an arbitrary base filed $k$.
Let us note that similar as above and to \cite{ElHoKhSoYa-MotDeloop} the arguments implies the result for an arbitrary pro-smooth base scheme.

\begin{lemma}
For an arbitrary filed $k$ the homomorphism ${\mathrm{K}}^\mathrm{MW}_n(k)\to \underline h^0(\mathbb ZF(\Delta\times -,\G_m^n))$ is surjective.
\end{lemma}
\begin{proof}
The claim follows similarly to \cite{Nesh-FrKMW} using proposition \ref{prop:MovingLemma} (moving lemma) proven in the next section, and separable field extension transfers for $\mathrm{K}^\mathrm{MW}_*$ form \cite[section 4,5]{M-A1Top} or \cite{CF}.
\end{proof}

\begin{lemma}
Assume one that one of the following conditions holds for a base scheme $S$
(a) $S = \Spec k$, $k$ is perfect, or
(b) the unramified Milnor-Witt K-theory $\unKMW_n$  over $S$ is strictly homotopy invariant for $n\geq 0$.

Then there is a homomorphism of sheaves $\underline h^0(\mathbb ZF(\Delta^\bullet,\G_m^{\wedge n})\to \underline{\mathrm{K}}^\mathrm{MW}_n$, for all $n\geq 0$, that the takes a correspondences $a\in Fr_1(pt,\G_m)$ defined by invertible $a\in k$ 
to the symbol $[a]$.
   
\end{lemma}
\begin{proof}[The proof for (a)]
The claim follows immediate form the universal property of the sheaf $\underline h^0(\mathbb ZF(\Delta^\bullet,\G_m^{\wedge n})$
since $\underline{\mathrm{K}}^\mathrm{MW}_n$ is a homotopy invariant stable linear presheaf with framed transfers.

Actually, is $\underline{\mathrm{K}}^\mathrm{MW}_n$ is the basic example of a presheaf with Milnor-Witt transfers, see \cite{FC-CWCorr}, 
in detail it is provided by the fact that $\underline{\mathrm{K}}^\mathrm{MW}_n$ is a zeroth homotopy groups of the complexes $C(X,G^n)$ \cite{Fasel-ChWittRing}, \cite{Fasel-GroupsdeCW}, the pushforwards for the homomologies of the complexes $C(X\times\PP^d,G^n,\mathcal O(1))_Z\to C(X,G^{n+d})_Z$, $Z\subset X\times\PP^n$ is closed finite over $X$, and the ring structure on the cohomologies of $C(X,G^{n+d})_Z$.
Hence $\underline{\mathrm{K}}^\mathrm{MW}_n$ is a stable framed presheaf because of the functor form the category of framed correspondences to the category of Chow-Witt correspondences constructed in \cite{DegFas} or \cite{DH}. 

Let us note that the homotopy invariance of $\underline{\mathrm{K}}^\mathrm{MW}_n$ follows form the isomorphism 
${\mathrm{K}}^\mathrm{MW}_n(\A^1_K)\simeq {\mathrm{K}}^\mathrm{MW}_n(K)$ due to the injectivity property for the framed stable linear homotopy invariant presheaves.
\end{proof}
\begin{proof}[The proof for (b)]
By the lemma \ref{cor:ReprKMW} the assumption implies that the $\G_m$-spectrum of Nisnevich sheaves $\mathcal{K}^\mathrm{MW}$
represents in $\mathbf D_{Nis,\A^1}(Sh_{Nis})[\G_m^{-1}]$ the sheaves $\unKMW_n$. Hence the sheaves $\unKMW_n$ are a homotopy invariant stable linear framed presheaves like as any $\SH$-representable presheaf. Now by the universal property of the sheaf $\underline{h}^0(\mathbb ZF(\Delta^\bullet\times -, \G_m^n)$ there is a homomorphism $\underline{h}^0(\mathbb ZF(\Delta^\bullet\times -, \G_m^n)\to \unKMW_n$ induced by the map $\G_m^n\to \unKMW_n\colon (a_1,\dots a_n)\to [a_1,\dots a_n]$.

It follows form the definitions that the composition $\mathrm{K}^\mathrm{MW}_n(k)\to \underline{h}^0(\mathbb ZF(\Delta^\bullet\times -, \G_m^n)\to \mathrm{K}^\mathrm{MW}_n(k)$ is identity for $n\neq 0$.
This proves the injectivity of the map $\mathrm{K}^\mathrm{MW}_n(k)\to \underline{h}^0(\mathbb ZF(\Delta^\bullet\times -, \G_m^n)$ for $n\geq 0$.

\end{proof}
\begin{lemma}\label{cor:ReprKMW}
(1) Assume that $\unKMW_n$ is strictly homotopy invariant for all integer $n\in \mathbb Z$ over some base $S$;
then the canonical homomorphisms $\unKMW_n\to \underline{\pi}^{n,n}_s(\mathcal K^\mathrm{MW})$ are isomorphisms for $n\in \mathbb Z$,
where $\mathcal K^\mathrm{MW}$ is
the spectrum of Nisnevich sheaves of abelian groups
$$\mathcal K^\mathrm{MW} = (\unKMW_0, \unKMW_1, \dots \unKMW_n\dots),\; \unKMW_{n}\times\G_m\to \unKMW_{n+1}\colon (\phi,a)\mapsto \phi \cdot a.$$

(2) Assume that $\unKMW_n$ is strictly homotopy invariant for all integer $n$ larger some $n_0$ over some base $S$ then $\unKMW_n$ is strictly homotopy invariant for all integer $n$;
\end{lemma}
\begin{proof}
(1) 
It follows form the strictly homotopy invariance of the sheaf $\unKMW_n$ that the fibrant replacement of $\unKMW_n$ with respect to the injective Nisnevich local model structure on the category of simplical specta of Nisnevich sheaves is $\A^1$-local.
So $[X, \unKMW_n]_{\SH_{S^1}(k)} = \unKMW_n$.
(Actually, in the case of a filed base case$\unKMW_n$ defines an element in hart of $\SH_{S^1}(k)$ with respect to the homotopy t-structure on $\SH_{S^1}(k)$ \cite[section 6.2]{M-Connect}.)

Then the claim follows form the isomorphisms  
\begin{equation}\label{eq:CancRKMW}
\unKMW_{n}(-)\simeq \unKMW_{n+1}(-\wedge\G_m)
\end{equation}
given by the canonical isomorphisms
$\unKMW_n(X\times\G_m)\simeq \unKMW_n(X)\oplus\unKMW_{n-1}(X)$.

(2) The claim follows immediate form \eqref{eq:CancRKMW}.
\end{proof}
\begin{remark}
Let us recall that the last isomorphism \eqref{eq:CancRKMW} follows form the homotopy invariance of $\unKMW_n$.
Consider the homomorphism $\unKMW_n(X\times\G_m)\simeq \unKMW_n(X)\oplus\unKMW_{n-1}(X)$ defined by the sub of the 
inverse image along the unit section $i_1\colon X\to X\times \G_m$, and the residue map at zero section
$\delta_0\colon \unKMW_n(X\times\G_m)\to \unKMW_{n-1}(X)$.
The inverse image $p^*$ along the projection $p\colon X\times\G_m\to X$ and $i_1^*$ induces the splitting $\unKMW_n(X\times\G_m)\simeq \unKMW_n(X)\oplus \Coker(p^*)$; 
on other side the morphism $\unKMW_{n-1}(X)\to \unKMW_n(X\times\G_m)\colon [a_1,\dots,a_{n-1}]\mapsto [t,a_1,\dots a_{n-1}]$ induces the left inverse to $\delta_0$, so we have the splitting 
$\unKMW_n(X\times\G_m)\simeq \Ker(\delta_0)\oplus \unKMW_{n-1}(X)$
Now the claim follows since
$Im(j^*) = \Ker(\delta_0)$, where $j\colon X\times\G_m\to X\times\A^1$, and
$Im(p^*)=Im(j^*)\simeq \unKMW_n(X)$.
\end{remark}
\begin{remark}
Let us note that the isomorphism $\unKMW_n(X)\simeq \unKMW_{n+1}(X\wedge \G_m)$ follows form the case of $X=pt$ and homotopy invariance of $\unKMW_n$ due to the injectivity for the framed linear stable homotopy invariant presheaves.
\end{remark}

\begin{theorem}
Assume one of the following (a) the base filed $k$ is perfect, (b) the base field $k$ is of characteristic different form 2. 
The homomorphisms of sheaves $\unKMW_n\to  \underline h^0(\mathbb ZF(\Delta\times -,\G_m^n))\to \underline{\pi}^{n,n}_s$ are an isomorphism and  $n\in \mathbb Z$. 
\end{theorem}
\begin{proof}
By the above lemmas we have the isomorphism $\mathrm{K}^\mathrm{MW}_n(k)\simeq \underline{h}^0(\mathbb ZF(\Delta^\bullet\times -, \G_m^n)$.
Then $\mathrm{K}^\mathrm{MW}_n(k)\simeq\underline{h}^0(\mathbb ZF(\Delta^\bullet\times -, \G_m^n)$ is an isomorphism for $n\geq 0$ due to the injectivity property for a homotopy invariant stable linear framed presheaves \cite{GP-HIVth}. 
 
Finally, the isomorphism 
$\unKMW_n\simeq \underline{h}^0(\mathbb ZF(\Delta^\bullet\times -, \G_m^n)$ for all $n$ 
follows from the 
the isomorphisms $\unKMW_n(X)\simeq \unKMW_{n+1}(X\wedge \G_m)$,
and the isomorphisms $\underline{h}^0(\mathbb ZF(\Delta^\bullet\times -, \G_m^n)\simeq \underline{\pi}^{n,n}_s$ and
$\underline{\pi}^{n,n}_s(\G_m)\simeq \underline{\pi}^{n-1,n-1}_s(pt)$.
\end{proof}


\subsection{The strictly homotopy invariance of $\unKMW_n$.}\label{subsect:SHIKMW}

In the subsection we summarise known arguments for the strictly homotopy invariance of $\unKMW_n$.

\subsubsection{Morel's pullback}
Firstly we recall the argument from \cite{M02} for the case of a field $k$, $\chark k\neq 2$.

\begin{lemma}
There are isomorphisms of sheaves $\unKMW_n\simeq {I}^n\times_{{I}^n/{I}^{n+1}} \underline{\mathrm{K}}^\mathrm{M}_n$ for all $n\in \mathbb Z$.
\end{lemma}
\begin{proof}
We refer to \cite{BM_PullBack}, \cite{Mor_PullBack} and \cite{GSZ-MWloc} for the case of the sections on fields of odd characteristic.
(Nevertheless the author haven't found a reference for the proof of the pullback of sheaves)
To get the claim for the sheaves firstly we need to note that all maps in the pullback square commutes with the residue morphisms, 
where by the residue morphism on Witt groups we mean the homomorphism $W(k(X))\to W(x)$ for $x\in X^(1)$ constructed by Schmid in \cite{Schmid-Wittringhom}. 
Namely this is provided by the formulas \cite[theorem 2.15]{M02} for the residues on Milnor-Witt K-theory, the similar formula of residue homomorphism on Milnor-K-theory \cite{Milnor70}, and for the residue map constructed by Schmid \cite[section 2.2, $\mathbf{D\tilde W}3$, formula bottom of page 21]{Schmid-Wittringhom}. 
Now it is enough to note that $W(U)\to W(k(U))$ is injective for a essential smooth $U$ and $W(U)\subset \Ker(W(k(U))\to \bigoplus_{x\in U^(1)}W(x))$.
\end{proof}
\begin{remark}
Let $Z\in X$ be a closed subscheme of codimension one in a local essentially smooth $k$-scheme $X$, $Z =Z(t)$, $t\in \mathcal O(X)$. 
Let $\delta\colon W(X-Z)\to W(Z)$ with respect to the equation $t$.
To check that the morphism $\mathrm{K}^\mathrm{MW}\to I^n$ is agreed wit the differentials it is enough to prove that $\delta(\langle t\rangle) = \langle 1\rangle$. In the case of a smooth scheme $Z$ it is given by the standard formula of the Gysin map. In an arbitrary case the claim follows form the regular one due to the rigidity along non-smooth closed embeddings for $W(-)$ proven by S.~Gille, since any such $X$ there is an embedding $X\subset X^\prime$ and $Z\subset Z^\prime$, $Z^\prime\to X^\prime$ is a closed smooth subscheme of codimension one.
\end{remark}

\begin{proposition}
Let the base be a filed $k$, $\chark k \neq 2$.
Then the sheaves $\unKMW_n$ are strictly homotopy invariant for $n\geq 0$.
\end{proposition}
\begin{proof}
Let us repeat some of the arguments form \cite[section 6, steps 1-4]{M02}.

The claim follows from the lemma above and from the strictly homotopy invariance of $underline{\mathrm{K}}^\mathrm{M}_*$, $I^n$ and $I^{n+1}/I^n$. The sheaves $\underline{\mathrm{K}}^\mathrm{M}_*$ and $\underline{\mathrm{K}}^\mathrm{M}_*/2$
are Rost's cyclomodules it follows form \cite[proposition 8.6, proposition 2.2(H)]{Rost-CycMod} and \cite[Theorem 6.1]{Rost-CycMod}
that $\underline{\mathrm{K}}^\mathrm{MW}_*$ and $\underline{\mathrm{K}}^\mathrm{MW}_*/2$ are strictly homotopy invariant.
So it follows form Milnor's conjecture $\underline{I}^n/\underline{I}^{n+1} = \underline{\mathrm{K}}^\mathrm{M}_n/2$, \cite{M_MilnorConj}, \cite{OVV-MilnorConj}, that the sheaves $\underline{I}^n/\underline{I}^{n+1} $ are strictly homotopy invariant. 
Since as proven in \cite{Horn} that the sheaves $W^n(-)$ are $\SH(k)$-representable, they are strickly homotopy invariant. 
Hence by induction we get strictly homotopy invariance of the sheaves $\underline{I}^n$.
%
\end{proof}

\subsubsection{Unramified sheaves and Milnor-Witt cyclomodules}

The next proof in the case of a perfect field of an arbitrary characteristic was given in \cite[Chapter 5]{M-A1Top}. The idea is to combine the theory of Rost cyclomodules and it's adaptation for Witt-groups to get the precise construction of one complex, a so called Rost-Smidt complex, defined over an arbitrary perfect filed
and  equal in the odd characteristic base field case to the fibred product of the Rost complex for $\mathrm{K}^\mathrm{MW}$ and the similar complex for Witt groups constructed by Schmid. 
In \cite{M-A1Top} the starting object which gives a rise to the complex is are unramified sheaves, and the main example is the Milnor-Witt K-theory. 

Recently, the idea was revisited and deeply studied in works \cite{NFELD-MWCycMod} and \cite{CF}. 
So called Milnor-Witt cyclomodules are defined, and the main example is the Milnor-Witt K-theory. 
It is proven in particular that the unramified sheaf corresponding to the Milnor-Witt cyclomodule is strictly homotopy invariant over a perfect field.

\subsubsection{Chow-Witt correspondences.}

Also one proof for the case of a perfect field is provided by the strictly homotopy invariance theorem for a homotopy invariant sheaves with Milnor-Witt transfers.
Actually, it follows form the definitions that the sheaves $\unKMW_n$ are 
sheaves with Chow-Witt correspondences, see \cite{FC-CWCorr}, so by \cite{DH} the homotopy invariance of $\unKMW_n$ implies the strictly homotopy invariance.

\section{Neshitov's moving lemma.}

%
%


\begin{definition}
Let $c= (Z,V;\phi,g)\in Fr_n(X, \G_m^n)$, $V\to \A^n_X$ is an etale neighbourhood of a closed subscheme $Z$ in $\A^n_X$, $Z$ is finite over $S$, $\phi\colon V\to \A^n$, $g\colon V\to \A^n$, $Z = V\times_{\phi,\A^n}0$.
We say that $c$ is \emph{simple} 
iff
$Z$ is smooth over $S$. 
\end{definition}

\begin{definition}
$\overline{Fr}_n(pt, \G_m^n)$ denotes the factor group of ${Fr}_n(pt, \G_m^n)$ up to $\A^1$-homotopy equivalence
\end{definition}

\begin{lemma}\label{lm:sectcor}
For any $c\in Fr_n(pt, \G_m^l)$ over an affine 
base scheme $S$ and for all large enough $d_i$, $i=1,\dots n$, and $r_j$, $j=1,\dots l$, 
there is 
a correspondence $c^\prime\in Fr_n(pt, \G_m^l)$ such that 
$[c^\prime]=[c]\in \overline{Fr}_n(pt, \G_m^l)$,
and such that
$$c^\prime = ( Z,\A^n-((Z(s)-Z)\cup Z(e)); s_1/t_\infty^{d_1}, \dots s_n/t_\infty^{d_n}; e_1/t_\infty^{r_1}, \dots e_l/t_\infty^{r_l}  )$$
for some 
sections $s_i\in \Gamma(\PP^n_S,\mathcal O(d_i)$, $d_i\in \mathbb Z$, $i=1,\dots n$, $e_j\in \Gamma(\PP^n_S,\mathcal O(r_j)$, $r_j\in \mathbb Z$, $j=1,\dots l$.
\end{lemma}
\begin{proof}

By Serre's theorem \cite[theorem 5.2]{Hartshorne-AlG} we can choose integers $d_i$ and sections $s=(s_i)$, $1\leq i \leq n$, $s_i\in \Gamma(\PP^n,\mathcal O(d_i))$, $s_i/t^{d_i}_\infty=\phi_i\big|_{Z(I(Z)^2)}$, $s_i\big|_{\PP^{n-1}}=t_\infty^{d_i}$, where $\P^{n-1}\subset \P^n$ is the subspace at infinity and $t_\infty\in \mathcal O(1)$, $Z(t_\infty)=\P^{n-1}$.
Similarly we can choose sections $e_i\in \Gamma(\PP^n,\mathcal O(l_i))$, $1\leq i\leq k$,
$e_i/t^{l_i}_\infty\big|_{Z(I(Z)^2)}=g_i\big|_{Z(I(Z)^2)}$,
where the $g_j$'s are the coordinates of the composition $\mathcal V\xrightarrow{g} Y \hookrightarrow\A^e$.
The functions $\lambda v^*(s_i/t_\infty^{d_i})+(1-\lambda)(\varphi_i)$ and $\lambda v^*(e_j/t_\infty^{r_j})+(1-\lambda)g_j$
gives a homotopy from $c$ to the framed correspondence 
$c^\prime=( Z,\A^n-(Z(s)-Z); s_1/t_\infty^{d_1}, \dots s_n/t_\infty^{d_n}; e_i/t_\infty^{l_i}, \dots (e_l/t_\infty^{r_l}  )$.
\end{proof}

\begin{notation}
Denote by $(F)_x$ the 
denote the fibre of the coherent sheaf $F$ on the scheme $X$ at a point $x\in X$, i.e. 
$(F)_x = i_x^*(F)$,  where $i_x\colon x\to X$ is the canonical embedding.

Denote $\hat N_{Y/X}$ the conormal sheaf of the closed subscheme $Y\subset X$. 
\end{notation}
\begin{lemma}\label{lm:XnsSubscheme}
Let $p\colon X\to Y$ be a finite morphism of schemes. 
Then there is a closed subscheme $X^{ns}\subset X$ such that 
$x\in X^{ns}$ iff
$x\in \Supp \Omega_p$, or
the residue filed at $x$ is not separable over the residue filed of the image of $x$ in $Y$;
%
Consequently if $p\colon X\to Y$ is flat then $x\in X^{ns}$ iff $p$ is not etale at $x$.
%
\end{lemma}
\begin{proof}
Consider the projection $p_1\colon X\times_{Y} X\to X$, which is finite morphism as well. 
Denote by $\Delta_X\subset \widetilde X\times_{Y} X$ the diagonal subscheme,
and fro any point $x\in X$ denote by $\delta_x\in \Delta$ the corresponding point under the canonical isomorphism $\Delta_X\simeq X$.
Then $$X^{ns} = p_1(\Supp \Omega_{p_1}\cap \Delta)\subset X.$$
Actually, let $x\in X$, denote $\widetilde X = X\times_Y x$, $x^2 = x\times_Y x$, and let $\delta_x\subset x\times x$ is the diagonal. Then 
the claim follows form 
the short exact sequence
$$0\to (\hat N_{x^2/\widetilde X})_x\to (\Omega_{p_1})_{\delta_x} \to \hat N_{\delta_x/x^2} ,$$
and isomorphisms
$$\hat N_{x^2/\widetilde X}\simeq p_2^*((\Omega_{p})_x),\, \hat N_{\delta_x/x^2}\simeq \Omega_{x^2\to x} ,$$
where $p_2\colon x^2\to x$ is the projection onto the second multiplicand.
\end{proof}

\begin{corollary}\label{cor:BigOpenetale}
Let $p\colon X\to Y$ be flat 
finite surjecive morphism, 
and $X$ is irreducible.
Assume that there is a point $x\in X$ such that the residue filed extension $\mathcal O(x)/\mathcal O(p(x))$ is separable, and $f$ is unramified at $x$.
Then there is a non-empty open subscheme $U\subset Y$ such that 
$X\times_Y U\to U$ is etale.
\end{corollary}
\begin{proof}
It follows form lemma \ref{lm:XnsSubscheme} and from assumption that there is a proper closed $X^{ns}\subset X$ such that $X-X^{ns}\to Y$ is etale. 
Since $X\to Y$ is finite surjective and $X$ is irreducible, then so is $Y$. Let $d=\dim Y$.
Since $X^{ns}\subset X$ is proper and since $X\to Y$ is finite, it follows that $\dim X^{nc}<d$.
Hence $Y-p(X^{ns})\neq \emptyset$.
Thus the claim is true for $U=Y-p(X^{ns})$.
\end{proof}

\newcommand{\der}{\operatorname{d}}

\begin{lemma}\label{lm:FinSur}
Let $S$ be a noetherian scheme of a finite type over $\mathbb Z$;
let $s_1,\dots s_n\in \Gamma(\PP^n_S,\mathcal O(d))$ be a set of sections, $s_i\big|_{\PP^{n-1}_S} = t_i^d$.
Then the vanishing locus $Z(s_1,\dots s_n)$ is finite surjective and flat over $S$.
\end{lemma}
\begin{proof}
Consider the morphism $f\colon \A^n\times\Gamma_d\to \A^n\times\Gamma_d$ 
defined by the regular functions  $\mathbf s_{i,d}/ t_\infty^d\in \mathcal O(\A^n\times\Gamma_d)$.
Since $\mathbf s_{i,d}\big|_{\PP^{n-1}_S} = t_i^d$, it follows that $f$ is quasi-finite.
In the same time the condition provides that the graph of $f$ is equal to the vanishing locus $Z(\mathbf s_1-\alpha_1 t_\infty^d, \dots \mathbf s_n-\alpha_n t_\infty^d)\subset \PP^n\times \A^n_S$ where $(\alpha_1,\dots \alpha_n)$ denotes coordinates on $\A^n$.
Hence $f$ is projective. Thus $f$ is finite, and since dimensions of the domain and the co-domain of $f$ are equal it follows that $f$ is finite.

Now let $x\in S$ be a point, $U\subset S$ is a affine Zariski neighbourhood of $x$. Since $U$ is affine there is a closed embedding 
$U\subset \Spec R$ be a regular ring $R$.
Consider a lift $\widetilde s_i\in \Gamma(\PP^n_{\Spec R}, \mathcal O(d))$ of the sections $s_i$, and the morphism 
$\widetilde f\colon \A^n_{\Spec R}\to \A^n_{\Spec R}$ defined by $\widetilde s_i/t^d_\infty$.
Then by the same reason as for $f$ the morphism $\widetilde f$ is finite and surjective.
Hence $\widetilde f$ is flat by \cite[Corollary 3.6]{AltKl}.
Thus $Z(\mathbf s_d)$ is flat over $\Gamma_d$.
\end{proof}
\begin{lemma}\label{lm:UniVLSmFlFsurCon}
Let $S$ be a scheme,
and denote by
$\mathcal O(1)$ the ample bundle on $\PP^n_S$ over $S$ 
and denote by $t_1,\dots t_n, t_\infty$ the coordinate section of $\mathcal O(1)$, in particular $Z(t_\infty) = \PP^{n-1}$ is the infinite hypersurface.
Assume that $Z\subset \A^n_X$ is a closed subscheme finite over $S$, 
$e\in \mathbb Z$, and 
$\beta_i\in \Gamma(Z(I^2(Z)),\mathcal O(e))$, $i=1,\dots n$ are sections such that $Z(\beta_1,\dots\beta_n) = Z$.

Denote by $\Gamma_d$ the affine space over $S$ that $S$-points is the set
$$\Gamma_d(S)=\{(s_1,\dots s_n)\in \Gamma(\PP^n_S, \mathcal O(d)^n)| 
s_i\big|_{Z(I^2(Z))} = \beta_i t_\infty^{d-e},
s_i\big|_{\PP^{n-1}_S} = t_i^d,
\},$$ 
Let $\mathbf s_d = (\mathbf s_{1,d},\dots \mathbf s_{n,d})\in \Gamma(X\times \Gamma_d, \bigoplus_i \mathcal L_i\otimes\mathcal O(d))$ be the universal section. 
Denote by $\mathbf Z_d$ the closed subscheme $\mathbf Z_d=Z(\mathbf s_d)- (Z\times\Gamma_d)\subset X\times\Gamma_d$.

Then 
there exist $N$ such that $\forall d>N$ the vanishing locus $\mathbf Z_d$ is connected and smooth over $S$,
and $\mathbf Z_d$ is flat finite surjective over $\Gamma_d$.
\end{lemma}
\begin{proof}

It follows form the relative version of the Serre's theorem \cite[theorem 8.8]{Hartshorne-AlG} that 
there is $N$ such that $\forall d>N$ the homomorphisms 
$\Gamma(X,\mathcal L_i\otimes\mathcal O(d) )\to \Gamma(Z(I^2(Z))\amalg x_1\amalg x_2,\mathcal L_i\otimes\mathcal O(d) )$, 
are surjective for all $i=1,\dots n$, $x_1,x_2\in X$ is a pair of different closed points.

Then for all $d>N$ the universal vanishing locus $Z(\mathbf s_d)$ is smooth over $S$.
Actually, let $s=(s_1,\dots s_n)\in \Gamma_d$ be an $S$-point, and let $x\in Z(s_1,\dots s_n)\subset X$.
By assumption there is a section $s^\prime \in \Gamma(Z(I^2(Z)),\mathcal O(d) )$ such that 
$s^\prime\big|_{Z(I^2(Z))\amalg \PP^{n-1}_S}=0$, and $s^\prime\big|_{x}$ is invertible.
Denote by \begin{equation}\label{eq:esubscpv}v_i = (0,\dots,0,s^\prime,0,\dots 0),\end{equation} where $s^\prime$ is located at the $i$-th slot, 
the vectors in the tangent space $T_{\Gamma_d,s}$.  
Now on the one side we have
\begin{equation}\label{eq:SmGraphVl}
Z(\mathbf s_d)\times_{\PP^n_S} (\PP^n_S-Z(s^\prime)) = Z(\mathbf s_{1,d}/s^\prime_{1}, \dots \mathbf s_{n,d}/s^\prime_{n})\times_{\PP^n_S} (\PP^n_S-Z(s^\prime)).
\end{equation}
On the other side
we see that differentials of the functions $\mathbf s_{1,d}/s^\prime_{1}$, $i=1,\dots n$, at the point 
$(x,s)\in X\times\Gamma_d$ in the directions defined by vectors $v_j$, $j=1\dots n$, are linearly independent,
namely
$$\der_{v_i} (\mathbf s_{i,d}/s^\prime_{i}) = 1, \der_{v_i} (\mathbf s_{i,d}/s^\prime_{i}) = 0, i\neq j.$$
Thus the conormal cone of $Z(\mathbf s_d)$ in $X\times \Gamma_d$ is a vector bundle of the dimension 
$n$.
So $Z(\mathbf s_d)$ is smooth over $S$, since $X$ and $\Gamma_d$ are smooth.

Now we need to show that $\mathbb Z_d$ is connected.
By lemma \ref{lm:FinSur} $\mathbb Z_d$ is flat finite and surjective over $\Gamma_d$. So it is enough 
to show that for any  $s\in \Gamma_d(S)$ and $x_1,x_2\in Z(s)$
there is a subspace $E\subset \Gamma_d$ such that 
$x_1$ and $x_2$ are in the same connected component of $\mathbb Z_d\times E$.
Consider the section 
$s^\prime\in \Gamma(\PP^n_X,\mathcal O(d))$, $s^\prime\big|_{PP^{n-1}_S\amalg Z(I^2(Z))}=0$, $s^\prime\big|_{x_1\amalg x_2}$ is invertible.
Define $E$ as the subspace of $\Gamma_d$ spanned by the point $s$ and tangent vectors $v_i$ \eqref{eq:esubscpv}.
Then we see from \eqref{eq:SmGraphVl} that $Z(\mathbf s_d)\times_{\PP^n_S} (\PP^n_S-Z(s^\prime))$ is equal to the graph of the map $(\PP^n_S-Z(s^\prime)\to \A^n_S$ given by regular functions $s_i / s^\prime$.
So it is connected. 
And by assumption on $s^\prime$, we have $x_1,x_2\in \PP^n_S - Z(s^\prime)$. Thus the claim follows.
%
%
\end{proof}

\begin{corollary}\label{cor:etalevanishloc}
Let $Z\in \PP^n_S$ be a closed subscheme in the projective space over a semi-local base scheme $S$ with infinite residue fields.
Let $\beta_i\in \Gamma(Z(I^2(Z)), \mathcal O(e))$, $i=1\dots n$, be a set of sections for some $e\in \mathbb Z$.
Then for all large enough $d$ there is a vector of sections $(s_1,\dots s_n)$, $s_i\in \Gamma(\PP^n_S,\mathcal O(d))$ such that
$s_i\big|_{Z(I^2(Z))}=\beta_i t_\infty^{d-e}$, $s_i\big|_{\PP^{n-1}_S}=t_i^d$ and such that
$Z(s_1,\dots s_n) - Z$ is etale over $S$. 
\end{corollary}
\begin{proof}
Consider the universal section $\mathbf s_d$ on $\PP^n\times\Gamma_d$ as in lemma \ref{lm:UniVLSmFlFsurCon}.
By Serre's theorem \cite[theorem 5.2]{Hartshorne-AlG} for all large enough $d$ there is a vector $s=(s_1,\dots s_n)\in \Gamma_d(S)$ such that
$s_i\big|_{Z(I^2(0_S))} = t_i t_\infty^{d-1}$ where $0_S\subset \PP^n_S$ denotes the zero point-section.
Then the morphism $Z(s_1,\dots s_n)-Z\to S$ is etale on $0_S$. 
Hence by corollary \ref{cor:BigOpenetale} there is a non-empty open subscheme $U\subset \Gamma_d$ such that
$\mathbf Z_d\times_{\Gamma_d} U\to U$ is etale.
Now since $S$ is semi-local with infinite residue fields,
there is an $S$-point $s\colon S\to U$. 
So $s$ is a vector of sections $(s_1,\dots s_n)$ such that $Z(s_1,\dots s_n) = \hat Z\amalg Z^s$ and $Z^s$ is etale over $S$.
\end{proof}

\begin{proposition}\label{prop:MovingLemma}
For any $c\in Fr_n(pt, \G_m^n)$ over a semi-local 
base scheme $S$ there are 
simple correspondences $c^+,c^-\in Fr_n(pt, \G_m^n)$ such that $[c^+]-[c^-]=[c]\in \overline{\mathbb ZFr}_n(pt, \G_m^n)$. 
\end{proposition}
\begin{proof}


We can assume that the residue fields of $S$ is infinite due to the finite descent for framed correspondences, see Appendix A, lemma  \ref{lm:LFrS}.
In details, assume the result for local schemes with infinite residue fields; then for an arbitrary $S$ we can consider extensions $S_{1,l}\to S$ and $S_{2,l}\to S$ defined by equations $x^{q_1^l}-1$ and $x^{q_2^l}-1$ on $S$, where $q_1q_2$ are prime integers coprime to characteristics of $S$, and $n\in mathbb N$. Let $S_1 = \varprojlim_l S_{1,l}$, $S_2 = \varprojlim_l S_{2,l}$. Then $S_1$ and $S_2$ are semi-local schemes with infinite residue fields, so by assumption the there are simple correspondences $c_i^+, c_i^-\in Fr(S_i,\G_m^n)$, $[p_i^*(c)]=[c_i^+]-[c_i^-]$, $p_i\colon S_i\to S$, $i=1,2$. Correspondences $c_i^*$ are defined by finite set of data over $\mathcal O(S_i)$ and hence $c_i^*$ are defined over $S_{i,l}$ for some $l\in \mathbb Z$. By assumption on $q_1$ and $q_2$ the schemes $S_{1,l}$ and $S_{2,l}$ are etale over $S$. Hence the correspondences $c^* = (c_1^*\amalg c_2^*)\circ L$ given by the finite descent, where $L\in Fr(S, S_{1,l}\amalg S_{2,l})$ is defined in lemma \ref{lm:LFrS}, are simple, and by lemma \ref{lm:LFrS} we have 
$[c] = [c^+]-[c^-]$.

By lemma \ref{lm:sectcor} we can assume
$$c = ( Z,\A^n-((Z(s)-Z)\cup Z(e)); s_1/t_\infty^{p}, \dots s_n/t_\infty^{d}; e_1/t_\infty^{q}, \dots e_l/t_\infty^{q}  )$$
where $s_i\in \Gamma(\PP^n_S,\mathcal O(p))$, $p\in \mathbb Z$, $i=1,\dots n$, 
$e_j\in \Gamma(\PP^n_S,\mathcal O(q))$, $q\in \mathbb Z$, $j=1,\dots l$, and $Z(e) =Z(e)$.
Denote $\hat Z= Z(s) - Z$, then $Z(s) = Z\amalg \hat Z$.

By corollary \ref{cor:etalevanishloc} we see that there is $N$ such that for all $d>N$ 
there are vectors of sections 
$$s^+=(s^+_1,\dots s^+_n)\in \Gamma(\PP^n_S,\mathcal O(d)^n),\;s^-=(s^-_1,\dots s^-_n)\in \Gamma(\PP^n_S,\mathcal O(d)^n)$$
such that 
$$\begin{array}{lcllcl}
s^+_i\big|_{Z(I^2(\hat Z))} &=& s_i\big|_{Z(I^2(\hat Z))} t_\infty^{d-p}, &
s^+_i\big|_{\PP^{n-1}_S} &=& t_i^{d},\\ 
s^-_i\big|_{Z(I^2(Z(s)))} &=& s_i\big|_{Z(I^2(Z(s)))} t_\infty^{d-p}, &
s^-_i\big|_{\PP^{n-1}_S} &=& t_i^{d}
\end{array}$$
and such that 
$Z(s^+)-\hat Z$ and $Z(s^-) - Z(s)$ are etale over $S$.

In the same times by Serre's theorem for all large enough $d$
there are sections $r_i\in \Gamma(\PP^n_S , \mathcal O(d-p) )$, 
$r_i\big|_{Z(I^2(Z(s))}=t_\infty^d$, 
$r_i\big|_{\PP^{n-1}_S}=t_i^d$.

Denote $s^\prime = (s^\prime_1,\dots s^\prime_n)$, $s^\prime_i = s_i r_i$.
Then the affine homotopy of framed correspondences given by 
$\lambda s_i r_i + (1-\lambda) s_i t_\infty^{d-p}$ implies that
$$c\stackrel{\A^1}{\sim} c^\prime=(Z, \A^n-( (Z(s^\prime)-Z)\cup Z(e), s_1r_1/t_\infty^{d},\dots s_n r_n/t_\infty^{d}; e_1/t_\infty^q, \dots e_l/t_\infty^q ).$$
On other side 
\begin{gather*}
c^\prime = c^{\prime,+} - c^{\prime,-}\in \mathbb ZFr_n(pt_S,\G_m^l),\\
\begin{array}{lcllll}
c^{\prime,+}=
(Z(s^\prime) -\hat Z, 
\A^n-( (Z(s^\prime)-\hat Z)\cup Z(e)); &
s^\prime_1/t_\infty^{d},\dots s^\prime_n/t_\infty^{d}; &e_1/t_\infty^q, \dots e_l/t_\infty^q ),\\
c^{\prime,-}=
(Z(s^\prime) - Z(s) , 
\A^n-( (Z(s^\prime)-Z(s) )\cup Z(e)); & 
s^\prime_1/t_\infty^{d},\dots s^\prime_n/t_\infty^{d}; & e_1/t_\infty^q, \dots e_l/t_\infty^q ).
\end{array}\end{gather*}
So the claim follows since by the above we have
$$\begin{array}{lcllll}
c^{\prime,+}&\stackrel{\A^1}{\sim}&
(Z(s^+)-\hat Z, &
\A^n-( (Z(s^+)-\hat Z)\cup Z(e));&
s^+_1/t_\infty^{d},\dots s^+_n/t_\infty^{d}; &
e_1/t_\infty^q, \dots e_l/t_\infty^q )
\\
c^{\prime,-}&\stackrel{\A^1}{\sim}&
(Z(s^-)-Z(s), &
\A^n-( (Z(s^-)-Z(s))\cup Z(e)); &
s^-_1/t_\infty^{d},\dots s^-_n/t_\infty^{d}; &
e_1/t_\infty^q, \dots e_l/t_\infty^q ).
\end{array}$$
\end{proof}

\section{Appendix A: the finite descent over a base.}

In this section we recall the $\A^1$-homotopy finite descent for framed correspondences and presheaves with framed transfers presented originally simultaneously and independently in \cite{DrKil-finfileds_AND_MorTh}\footnote{The finite descent for framed correspondences is written for the case of fields and representable presheaves but the finite descent we use here is given by the same formulas.} (see also \cite[Appendix]{FD-MotPairs}), and \cite[Appendix B]{ElHoKhSoYa-MotDeloop}. 
We refer to \cite{Voevodsky-FrCor} and \cite{GP_MFrAlgVar}
the theory of framed correspondences and framed motives,
see \cite[definition 2.1, definition 8.4]{GP_MFrAlgVar} for the definition of framed correspondences. 
We will use the functor form the category $\mathbb ZFr(S)\to \SHd(S)$ induced by the composition map $Fr_n(X,Y)\to Sh_{nis}(X\times\PP^n/ X\times \PP^{n-1}, Y\times \A^n/ Y\times(\A^n-0) )\to [X,Y]_{\SHd(S)}$.

\begin{proposition}\label{prop:FinDesc}
Let $S_1 = Z(f_1)\subset \A^1_S$ and $S_2 = Z(f_2)\subset \A^1_S$, where $f_1,f_2$ are polynomials of coprime degress with coefficients in the ring of regular functions on a scheme $S$ and leading coefficients being equal to 1.
Suppose that $S_1\to S$ and $S_2\to S$ are etale. 
Then the homomorphism $e\colon [X,Y]_{\SHd(S)} \to [X\times_S \widetilde S,Y\times_S\widetilde S]_{\SHd(S)}$ is injective for any dotted smooth schemes $X$ and $Y$ over $S$. 
\end{proposition}
\begin{proof}
To get the claim it is enough to construct the left inverse $e$.

Let $L$ be the morphism in $[S, \widetilde S]_{\SHd(S)}$ given by the sum of framed correspondence 
$(A^1_{S_1}-(S_1\times_S S_1 - \Delta_{S_1/S}) ,f,pr_{S_1}) \in Fr_1(S,S_1)$,
via the functor $\mathbb ZFr_*(S)\to \SHd(S)$,
where $\Delta_{S_1/S}\to S_1\times_S S_1$ is the diagonal, $pr_{S_1}\colon \A^1_{S_1}\to S_1$.
Let $p\colon \widetilde S\to S$ and $p_Y\colon Y\times\widetilde S\to Y$ be the canonical projections.
Let's denote $L_X=id_X=\boxtimes L= pr_{\#} pr^*(L)$, where $p\colon X\times \to S$ is the structural morphism.

Then the explicit framed correspondence $(\A^1_S\times \A^1 - (S_1\times_S S_1 - \Delta_{S_1/S})\times 0, Z(h), h, pr_{S_1\times\A^1} )$, where $h = (1-\lambda) f_{i}+\lambda x^{\deg f_i}$, $i=1,2$, gives 
the homotopy between $p \circ L.$ and the morphism defined by framed correspondence $\Lambda_{\deg f_i}$ defined by the framed correspondence $(\A^1_S, Z(x^{\deg f_i}, x^{\deg f_i}, pr_S)$, where $pr\colon \A^1_S\to S$ is canonical projection.
In the same time the homotopy given by $(\A^1_S, Z((1-\lambda) x^l + \lambda (x^l)(x-1), pr_S)$ gives the homotopy between $\Lambda_l$ and $\Lambda_{l-1}+\langle (-1)^{l-1}\rangle$, where $\langle a \rangle$ denotes the element in $[pt,pt]_{\SHd(S)}$ defined by the multiplication $\G^1\to \G^1\colon x\mapsto ax$ for any invertible regular function $a$ on $S$. 

Then the left inverse is given by 
$$\begin{array}{ccc}
\,[X\times_S \widetilde S,Y\times_S\widetilde S]_{\SHd(S)} &\to& [X,Y]_{\SHd(S)}\\
a &\mapsto & p_Y\circ a\circ L_X
\end{array}$$
Actually let $\widetilde a = p^*(a)$ be the base change of $a$ along the morphism $p\colon \widetilde S\to S$ for $a\in [X,Y]_{\SHd(S)}$, then
we have $$ a=a \circ p_X\circ L_X=p_Y\widetilde a\circ L_X,$$
since $p \circ L$ is $\A^1$ homotopy equivalent to the identity morphism $id_S$.
\end{proof}
 
We see form the above proof the following
\begin{lemma}\label{lm:LFrS}
For any base scheme $S$ and etale coverings $S_1\to S$, $S_2\to S$ defined by two separable polynomials over $S$ with unit leading terms and of a coprime degrees, 
there is framed correspondences $L\colon Fr_1(S,S_1\amalg S_2)$ such that $[p\circ L] = [id_S]\in \overline{\mathbb ZF}(S,S)$, where $\overline{\mathbb ZF}(S,S)$ is a factor sheaf of ${\mathbb ZF}(S,S)$ with respect to $\A^1$-homotopies.
\end{lemma}

\section{Appednix B: The sign for the compositions in $[\G_m^{\wedge l},\G_m^{\wedge n}]$.}
\begin{proposition} \label{prop:composa}
Let 
$f\in [\G_m^{\wedge n}, \G_m^{\wedge m^\prime}]_\SH$
$g\in [\G_m^{\wedge m}, \G_m^{\wedge n^\prime}]_\SH$, then
\begin{itemize}
\item[(1)]
$\Sigma_{\G_m}^{n^\prime} f \circ \Sigma_{\G_m}^{n} g \sim^\G  \langle -1\rangle^s \bullet(\Sigma_{\G_m}^{n^\prime} f \bullet \Sigma_{\G_m}^{n} g)$, where
$s=(m^\prime+n^\prime)(m^\prime+n)$;
\item[(2)]
$\Sigma_{\G_m}^{n^\prime} f \circ \Sigma_{\G_m}^{n} g \sim^{\G}
\langle -1\rangle^{s}\bullet
(\Sigma_{\G_m}^{m^\prime} g \circ \Sigma_{\G_m}^{m} f),$
where $s = (n^\prime+m)(m^\prime+n)+n^\prime+m$
\end{itemize}
\end{proposition}\begin{proof}[Proof of the proposition] \item[(1)] The equality is provided by the permutation on the middle term of the composition and lm \ref{lm:Twist}
\begin{multline*}
\Sigma^{1+n+n^\prime}_\G(\Sigma_{\G}^{n^\prime} f \circ \Sigma_{\G}^{n} g) = 
(\Sigma_{\G}^{1+(n+n^\prime)+n^\prime} f)  \circ (\Sigma_{\G}^{1+(n+n^\prime)+n} g)=\\
(\Sigma_{\G_m}^{1+n^\prime} f \Sigma_{\G_m}^{n+n^\prime}) \circ (\check P \circ \Sigma_{\G_m}^{1+(n+n^\prime)+n} g  \circ \hat P)
\stackrel{lm \ref{lm:Twist}}{=}\\
\Sigma^{m^\prime +2n^\prime+n}_\G\langle -1\rangle^{s} \circ (\Sigma_{\G_m}^{1+n^\prime} f \Sigma_{\G_m}^{n+n^\prime}) \circ (\Sigma_{\G_m}^{1+(n+n^\prime)+n} g)=\\
\langle -1\rangle^{s} \bullet ((\Sigma_{\G_m}^{1+n^\prime} f \Sigma_{\G_m}^{n+n^\prime}) \circ (\Sigma_{\G_m}^{1+(n+n^\prime)+n} g))=
\langle -1\rangle^{s} \bullet [\Sigma_{\G_m}^{n^\prime} f \bullet \Sigma_{\G_m}^{n} g]
\end{multline*} 
where 
$\hat P\colon \G_m^{m+n}\wedge\G_m^{n+n^\prime+1}\to \G_m^{n+n^\prime+1}\wedge\G_m^{m+n}$,
$\check P\colon \G_m^{n^\prime+n}\wedge\G_m^{n+n^\prime+1}\to \G_m^{n+n^\prime+1}\wedge\G_m^{n^\prime+n}$  
are the permutations which replace the multiplicands, $\sign \check P = (n^\prime+n)(n+n^\prime +1)=0$, $\sign \hat P= (n+n^\prime+1)(n+m)=s$.

\item[(2)]
Since 
$$\check G \circ (\Sigma_{\G_m}^{n^\prime} f \bullet \Sigma_{\G_m}^{n} g)\circ \hat G=
\Sigma_{\G_m}^{m^\prime} g  \bullet \Sigma_{\G_m}^{m} f,$$
where 
$$\begin{array}{cccc}
\check G\colon &
\G_m^{\wedge n}\wedge\G_m^{\wedge n^\prime}\wedge\G_m^{\wedge m}\wedge\G_m^{\wedge n}&\to& 
\G_m^{\wedge m}\wedge\G_m^{\wedge n}\wedge\G_m^{\wedge n}\wedge\G_m^{\wedge n^\prime},\\ 
\hat G\colon &
\G_m^{\wedge m^\prime}\wedge\G_m^{\wedge n^\prime}\wedge\G_m^{\wedge n^\prime}\wedge\G_m^{\wedge n}&\to& 
\G_m^{\wedge n^\prime}\wedge\G_m^{\wedge n}\wedge\G_m^{\wedge m^\prime}\wedge\G_m^{\wedge n^\prime},
\end{array}$$
and 
$\sign 
(\check G) \sign (\hat G) = (n+n^\prime)(m+m^\prime+1)$ 
the claim follows from point (1) and lm \ref{lm:Twist}.
\end{proof}

\begin{corollary}\label{cor:composa} Let 
$f\in [\G_m^{\wedge n}, \G_m^{\wedge m^\prime}]_\SH$
$g\in [\G_m^{\wedge m}, \G_m^{\wedge n^\prime}]_\SH$, then
\item[(0)]
$f\bullet g\sim^\G (\Sigma^1_\G f)\bullet g \sim^\G \langle -1\rangle^{m+n^\prime} (\Sigma^1_\G f)\bullet g$  
\item[(1)]
$ \Sigma^{1+n}(\Sigma_{\G_m}^{k+n^\prime} f \circ \Sigma_{\G_m}^{k+n} g) =  \langle -1\rangle^{k(n+n^\prime)}(\Sigma_{\G_m}^{k+n^\prime} f \bullet \Sigma_{\G_m}^{k+n} g) $ for any $k\in \mathbb Z$ such that all terms in the formula are defined.
\end{corollary}
\begin{proof}
(0) The claim follows form lemma \ref{lm:Twist}; (1) The claim follow form point (0) and prop \ref{prop:composa}.(2);
\end{proof}

\Addresses

\end{document}